\documentclass[10pt,twoside]{siamltex}
\usepackage{amsfonts,epsfig}

\newtheorem{remark}[theorem]{Remark}
\newtheorem{algorithm}[theorem]{Algorithm}

\newcommand{\integers}{\mathbb{Z}}
\newcommand{\mod}{\; (\mbox{mod}\;}
\newcommand{\D}{\displaystyle}

\begin{document}



\bibliographystyle{plain}
\title{
Embedding cocyclic D-optimal designs in cocyclic Hadamard matrices
}

\author{
V\'ictor \'Alvarez\thanks{Department of Applied Math I, University of Seville, Avda. Reina Mercedes s/n, 41012 Seville, Spain
(\{valvarez,armario,mdfrau,gudiel\}@us.es). }
\and
Jos\'e Andr\'es Armario\footnotemark[1]
\and
Mar\'ia Dolores Frau\footnotemark[1]
\and
F\'elix Gudiel\footnotemark[1]}

\pagestyle{myheadings}
\markboth{V.\ \'Alvarez, J.A.\ Armario, M.D.\ Frau, and F.\ Gudiel}{Embedding cocyclic D-optimal designs in cocyclic Hadamard matrices}
\maketitle

\begin{abstract}
In this paper a method for embedding cocyclic submatrices with ``large'' determinants of orders $2t$    in certain cocyclic Hadamard matrices of orders $4t$ is described ($t$ an odd integer). If these determinants attain the largest possible value, we are embedding D-optimal designs. Applications to the pivot values that appear when Gaussian Elimination with complete pivoting is performed on these cocyclic Hadamard matrices are studied.
\end{abstract}

\begin{keywords}
D-optimal Designs, Cocyclic Hadamard matrices, Embedded matrices, Gaussian elimination pivots.
\end{keywords}
\begin{AMS}
 05B20, 15A15, 65F40, 65F05.
\end{AMS}

\section{Introduction} \label{intro-sec}
A {\it Hadamard matrix} $H$ of order $n$ is an $n\times n$ matrix with elements $\pm 1$ and $H H^T=nI$. A Hadamard matrix is said to be normalized if it has its first row and column all $1's$. We can always normalize a Hadamard matrix by multiplying rows and columns by $-1$. These matrices must have order 1, 2 or a multiple of 4. It is conjectured that Hadamard matrices exist for every $n \equiv 0 \mod 4)$. Although no proof of this fact is known, there is much evidence about its validity (see \cite{Hor07} and the references there cited).

Two Hadamard matrices $H_1$ and $H_2$ are called {\em equivalent} (or Hadamard equivalent, or $H$-equivalent) if one can be obtained from the other by a sequence of row and/or column interchanges and row and/or column negations. The  question of classifying  Hadamard matrices of order $n\geq 32$ remains unanswered and only partial results are known \cite{Kou}.

Problems involving Hadamard matrices sound very easy, but they are notoriously difficult to solve. One interesting open problem, among others, is the question of the largest pivot encountered during the process of Gaussian Elimination (GE) with complete pivoting  for an $n\times n$ Hadamard matrix $H$ (the so called ``growth factor'' for $H$). Traditionally, backward error analysis for Gaussian Elimination (GE), see e.g. \cite{DP88}, on a matrix  $A=\left[a_{ij}^{(1)}\right]$ is expressed in terms of the {\em growth factor}

$$g(n,A)=\frac{\mbox{max}_{i,j,k}\;|a_{ij}^{(k)}|}{\mbox{max}_{i,j}\;|a_{ij}^{(1)}|}$$
which involves all the elements $a_{ij}^{(k)},\,k=1,2,\ldots,n$, that occur during the elimination for a choice of pivoting strategy given. Matrices with the property where no row and column exchanges are needed during GE with complete pivoting are called {\em completely pivoted} (CP) or feasible. In other words, at each step of the elimination the element of largest magnitude (the ``pivot'', denoted by $p_k$) is located at the top left position of every appearing submatrix during the process. If  $A(k)$ denotes the absolute value of  $k\times k$ principal minor of $A$, then mathematically $A$ being CP means (or is equivalent to) that for each $k$, we have that $A(k)$ is greater than or equal to the absolute value of any other $k\times k$ determinant that includes the first $k-1$ rows and columns. This is not necessarily the maximum $k\times k$ minor of $A$, but only the maximum $k\times k$ minor of $A$ when its first $k-1$ rows and columns are fixed.

For a CP matrix $A$ we have
$$g(n,A)=\frac{\mbox{max}\{p_1,p_2,\ldots,p_n\}}{|a_{11}^{(1)}|}.$$
If a matrix is not initially CP, by applying row and column operations with complete pivoting we can always bring it in CP form.

The following lemma gives a useful relation between pivots and minors.

\begin{lemma}\cite{Cry68} Let $A$ be a CP matrix. The magnitude of the pivots which appear after application of GE operations to $A$ is given by
$$p_j=\frac{A(j)}{A(j-1)},\quad j=1,2,\ldots, n,\,\,A(0)=1.$$
\end{lemma}

 In 1969, Cryer  \cite{Cry68} conjectured that if $A$ is a real $n\times n$ matrix such that $|a_{i,j}|\leq 1,$ then $g(n,A)\leq n$, with equality iff $A$ is a Hadamard matrix. In 1991 Gould \cite{Gou91} proved that the first part of the conjecture is not true. He found matrices with growth bigger than their orders. Thus, the following remains open:

\noindent{\bf Conjecture}(Cryer) The growth of a Hadamard matrix is its order.

This conjecture has only been proven for $n=4, 8, 12$ and $16$ (see \cite{Cry68,EM95,KM09}). Great difficulty arises in the study of this problem because $H$-equivalence operations do not preserve pivots, i.e. the $H$-equivalent matrices do not necessarily have the same pivot pattern. For instance, for $n=16$ there are 34 pivot patterns although there are only 5 equivalence classes of Hadamard matrices for this order. Furthermore, many pivot patterns can be observed by permuting the rows and columns of  any 20 by 20 Hadamard matrix and there are just 3 inequivalent matrices.

However, the existence of D-optimal designs (and other specific submatrices with concrete determinants) that exist embedded in a Hadamard matrix have  provided some clues on the pivot patterns  (see \cite{SM10,Mit11}).

A {\it D-optimal design of order $n$} is an $n\times n\,\,(1,-1)$-matrix having maximal determinant. Here and throughout this paper, for convenience, whenever  a determinant or minor is mentioned, we mean its absolute value.
The question of finding the  determinant of a D-optimal design of order $n$  is an  old one which remains unanswered in general.

In 1893 Hadamard proved in \cite{Had93} that for every $(-1,1)$-matrix $M$, \begin{equation} \label{det0} det(M) \leq n ^\frac{n}{2}. \end{equation}

We recall that the original interest in Hadamard matrices stemmed from the fact that these matrices are the only ones that satisfy equality in (\ref{det0}).

This has led to further study and
tighter  bounds for the maximal  determinant for all $(-1,1)$-matrices of order $n \neq 0 \mod 4)$ have been found (see \cite{Bar33,Ehl64,Ehl64b,Woj64,KMS00}). For instance, when $n \equiv 2 \mod 4)$, Ehlich  in \cite{Ehl64} and independently Wojtas in \cite{Woj64} proved that \begin{equation} \label{det2} det(M) \leq (2n-2) (n-2) ^\frac{n-2}{2}. \end{equation} In order for equality to hold, it is required that there exists a $(-1,1)$-matrix $M$ of order $n$ such that $MM^T=\left[ \begin{array}{cc} L&0\\ 0&L \end{array}\right]$, where $L=(n-2)I_{\frac{n}{2}}+2J_{\frac{n}{2}}$. Here, as usual, $I_n$ denotes the identity matrix of order $n$, and $J_n$ denotes the $n\times n$ matrix all of whose  entries are equal to one. In these circumstances, it may be proven that, in addition, $2n-2$ is the sum of two squares,  a condition which is believed to be sufficient  (order 138 is the lowest for which the question has not been settled yet, \cite{FKS04}).

In the early 90s, a surprising link between homological algebra and Ha\-da\-mard matrices \cite{HD94} led to the study of cocyclic Hadamard matrices \cite{HD95}.
Hadamard matrices of many types are revealed to be (equivalent to) cocyclic matrices \cite{DH93,Hor07}. Among them, Sylvester Hadamard matrices, Williamson Hadamard matrices, Ito Hadamard matrices and Paley Hadamard matrices. Furthermore, the
     cocyclic construction is the most uniform construction technique for Hadamard matrices currently known, and
     cocyclic Hadamard matrices
    may consequently provide a uniform approach to the
famous Hadamard conjecture.

The main advantages of the cocyclic framework concerning $4t$ by $4t$ Hadamard matrices may be summarized in the following facts:

\begin{itemize}

\item The  test to decide whether a cocyclic matrix is Hadamard runs in $O(t^2)$ time, better than the $O(t^3)$ algorithm for usual (not necessarily cocyclic)  matrices.

\item The search space is reduced to the set of cocyclic matrices over a given group (that is, $2^s$ matrices, provided that a basis for
cocycles over $G$ consists of $s$ generators), instead of the whole set of $2^{16t^2} $ matrices of order $4t$ with entries in $\{-1,1\}$.  \end{itemize}

In \cite{AAFG11} was shown that  the cocyclic technique can  certainly be extended to handle the maximal determinant problem at least when $n\equiv 2 \mod 4)$. More concretely, the study  focused on cocyclic matrices over the dihedral group of $2t$ elements  with $t$ odd. Based on exhaustive and heuristic searches,  three algorithms for constructing cocyclic matrices with large determinants  were provided.

In this paper we are interested in embedding (cocyclic) submatrices of orders $2t$ with large determinants in certain cocyclic Hadamard matrices of orders $4t$. If these determinants attain the largest possible value, we are embedding D-optimal designs. Also, we discuss   the relation between the existence of these submatrices and the growth factor for these Hadamard matrices.

In Section 2,  an algebraic formalism (in terms of cocycles) to describe two combinatorial operations on a matrix (eliminate and add certain rows and columns) is provided. As a consequence of this formalism  a method arises for embedding (cocyclic) submatrices of orders $2t$ with large determinants in certain cocyclic Hadamard matrices of orders $4t$. In Section 3, we connect the existence of specific matrices embedded in cocyclic Hadamard matrices of order 20 with  the values of the pivots that  appear when we perform Gaussian elimination with complete pivoting on them. The last section is devoted to conclusions and future work.

\noindent{\bf Notation.} Throughout this paper we use $-$ for $-1$ and $1$ for $+1$. We write $H$ for a Hadamard matrix and $D_j$ for a D-optimal design of order $j$. The notation $D_j \in H$ means $D_j$ is embedded in $H$.

\section{Cocyclic D-optimal designs embedded in Cocyclic Hadamard matrices}
Assume throughout that $G=\{g_1=1,\,g_2,\ldots,g_{n}\}$ is a multiplicative group, not necessarily abelian. Functions
$\psi\colon G\times G\rightarrow \langle -1\rangle\cong {\bf Z}_2$ which satisfy
\begin{equation}\label{condiciondecociclo}
\psi(g_i,g_j)\psi(g_ig_j,g_k)=\psi(g_j,g_k)\psi(g_i,g_jg_k), \quad\forall g_i,g_j,g_k\in G
\end{equation}
are
called (binary) cocycles  (over $G$) \cite{McL95}. A cocycle is a  coboundary
$\partial\phi$ if it is derived from a set mapping $\phi\colon
G\rightarrow \langle -1\rangle$  by
$\partial\phi(a,b)=\phi(a)\phi(b)\phi(ab)^{-1}.$

A cocycle $\psi$ is naturally
displayed as {\it a cocyclic matrix (or $G$-matrix)} $M_\psi$; 
 that is, the entry in the $(i,j)$th position of the cocyclic matrix is $\psi(g_i,g_j)$, for all $1\leq i,j\leq n$.


A cocycle $\psi$ is {\it normalized} if $\psi(1,g_j)=\psi(g_i,1)=1$ for all $g_i,g_j \in G$. The cocyclic matrix coming from a normalized cocycle is called {\it normalized} as well.  Each unnormalized cocycle $\psi$ determines a normalized one $-\psi$, and vice versa. Therefore, we may reduce, without loss of generality, to the case of  normalized cocycles.


The set of cocycles forms an abelian group $Z(G)$ under pointwise multiplication, and the coboundaries form a subgroup $B(G)$. A basis $ \cal B$ for cocycles over $G$ consists of some elementary coboundaries $\partial_i$ and some representative cocycles, so that every cocyclic matrix admits a unique representation as a Hadamard (pointwise)  product $M=M_{\partial_{i_1}}\circ \ldots\circ M_{\partial_{i_w}}\circ R$, in terms of some coboundary matrices $M_{\partial_{i_j}}$ and a matrix $R$ formed from representative cocycles.

Recall that every {\em elementary coboundary} $\partial_d$ is constructed from the characteristic set map  $\delta_d\colon G\rightarrow  \{-1,1\}$ associated with an element $g_d\in G$, so that
$$\partial_d(g_i,g_j)=\delta_d(g_i)\delta_d(g_j)\delta_d(g_ig_j)\quad \mbox{for}\quad \delta_d(g_i)=\left\{\begin{array}{rr} -1 & g_d=g_i,\\
1 & g_d\neq g_i.
\end{array}\right.$$

\begin{remark} \label{notacob} (\cite[Lemma 1]{AAFR08})

In particular, for $d\neq 1$, every row $s \notin \{ 1,d\}$ in $M_{\partial _d}$ contains precisely two $-1$s, which are located at the positions $(s,d)$ and $(s, e)$, for $g_e = g^{-1}_s g_d$. Furthermore, the first row is always formed by $1$s, while the $d$-th row is formed all by $-1$s, excepting the positions
$(d, 1)$ and $(d, d)$.
\end{remark}

Although the elementary coboundaries generate the set of all coboundaries, they might not be linearly independent (see \cite{AAFR09} for details).

Let $G_r(M)$ (resp. $G_c(M)$) be the Gram matrix of the rows (resp. columns) of $M$,
$$G_r(M)=MM^T,\quad (\mbox{resp.} \,G_c(M)=M^TM).$$
The Gram matrices of a cocyclic matrix can be calculated as follows.
\begin{proposition}(\cite[lemma 6.6]{Hor07}) \label{prop1} \newline
Let $M_\psi$ be a cocyclic matrix,
\begin{equation}\label{aatc}[G_r(M_\psi)]_{ij}=\psi(g_ig_j^{-1},g_j)\sum_{g\in G}\psi(g_ig_j^{-1},g),
\end{equation}
\begin{equation}
\label{atac}
[G_c(M_\psi)]_{ij}=\psi(g_i,g_i^{-1}g_j)\sum_{g\in G}\psi(g,g_i^{-1}g_j).
\end{equation}
\end{proposition}

If a cocyclic matrix $M_\psi$ is Hadamard, we say that the cocycle involved, $\psi$, is orthogonal and $M_\psi$ is {\it a cocyclic Hadamard matrix}. The cocyclic Hadamard test asserts that a normalized cocyclic matrix is Hadamard if and only if every  row sum (apart from the first) is zero \cite{HD95}. In fact, this is a straightforward consequence of Proposition \ref{prop1}.

Analyzing this relation from a new perspective, one could think of normalized cocyclic matrices meeting the bound (\ref{det0}) as normalized cocyclic matrices for which every row sum is zero. Could it be possible that such a relation translates somehow to the case $n\equiv 2 \mod 4)$? We now prove that, in fact, the answer to this question is affirmative.

A natural way to measure if the rows of a normalized cocyclic matrix $M=[m_{ij}]$ are close to sum zero, is to define an {\it absolute row excess} function $RE$, such that
$$RE(M)=\sum_{i=2}^n \left| \sum_{j=1}^n m_{ij} \right|.$$
This is a natural extension of the usual notion of {\em excess} of a Hadamard matrix, $E(H)$, which consists in the summation of the entries of $H$.

With this definition at hand, it is evident that a cocyclic matrix $M$ is Hadamard if and only if $RE(M)=0$. That is, a cocyclic matrix $M$ meets (\ref{det0}) if and only if $RE(M)$ is minimum. This condition may be generalized to the case $n \equiv 2 \mod 4)$.

For the remainder of the paper $t$ denotes an odd positive integer.

\begin{proposition} \cite{AAFG11}\label{prophadi}
Let $M$ be a normalized cocyclic matrix over $G$ of order $n=2t$. Then
\newline $RE(M)\geq 2t-2$.
\end{proposition}

But we may go even further. Having the minimum possible value $2t-2$ is a necessary condition for a cocyclic matrix $M$ to meet the bound $(\ref{det2})$.

\begin{proposition}\cite{AAFG11} \label{propvic}
If a cocyclic matrix $M$ of order $n=2t$ meets the bound (\ref{det2}), then $RE(M)=2t-2$.
\end{proposition}

Unfortunately, although having minimum absolute row excess is a necessary and sufficient condition for meeting the bound (\ref{det0}), it is just a necessary (but not sufficient, in general, see \cite[Table 5]{AAFG11}
) condition for meeting the bound (\ref{det2}). But there are some empirical evidences that matrices having minimum absolute row excess correspond with matrices having large determinants, see Table 2.1., page 11.

From now on, we fix $G={\cal D}_{2m}$ as the dihedral group with presentation $\langle a,b\colon \,a^m=b^2=(ab)^2=1\rangle$, with ordering $\{1,a,\ldots,a^{m-1},b,ab,\ldots,a^{m-1}b\}$  and indexed as $\{1,\ldots,2m\}$.
  A basis for cocycles over ${\cal D}_{2m}$ consists in (see \cite{AAFG11,AAFR08}):

   \begin{itemize}
   \item Let $m$ be  an odd positive integer.
   $${\cal B}=\{ \partial _2, \ldots , \partial_{2m-1},\beta_2\}.$$
   \item Let $m$ be an even positive integer.
   $${\cal B}=\{ \partial _2, \ldots , \partial_{2m-2},\beta_1,\beta_2,\gamma\}.$$
   \end{itemize}

  Here $\partial_i$ denotes the coboundary associated with the $i^{th}$-element of the dihedral group ${\cal D}_{2m}$, that is $a^{i-1 \, (\mbox{\tiny mod} \, m)}b^{\lfloor \frac{i-1}{m}\rfloor}$. And $\beta_1,\beta_2$ and $\gamma$ are the representative cocycles in cohomology, i.e. the cocyclic matrices coming from inflation are $M_{\beta_1}= J_m\otimes \left[ \begin{array}{rr} 1 & 1 \\  1 & - \end{array}\right]$ and $M_{\beta_2}=\left[ \begin{array}{rr} 1 & 1 \\ 1 & - \end{array}\right]\otimes J_m$.
  We use $A\otimes B$ for denoting the usual Kronecker product of matrices, that is, the block matrix whose blocks are $a_{ij}B$.

  The transgression cocyclic matrix $M_\gamma$ is $M_\gamma=\left[\begin{array}{rr} A_m & A_m \\ B_m & B_m \end{array}\right]$ for the $m\times m$ matrices
  $A_m=\left(a_{ij}\right)$ and $B_m=\left(b_{ij}\right)$ where
  $$a_{ij}=\left\{\begin{array}{rr} -1 & i+j>m+1 \\ 1 & \mbox{otherwise} \end{array}\right. ,\qquad\mbox{and}\qquad b_{ij}=\left\{\begin{array}{rr} -1 & i<j \\ 1 & \mbox{otherwise} \end{array}\right. ;$$
It has been observed that cocyclic Hadamard matrices over the dihedral group mostly use $M_{\beta_2}\circ M_{\gamma}$ and do not use $M_{\beta_1}$ (see \cite{AAFR09,Fla97}). In the sequel, we consider only cocyclic Hadamard matrices of this form $M_{{\partial}_{i_1}}\circ \ldots\circ M_{{\partial}_{i_w}}\circ M_{{\beta_2}}\circ M_{\gamma}$.

In what follows, the goal  is to provide an algebraic formalism (in terms of cocycles) to describe two combinatorial operations on a matrix:
the first consisting in eliminating and the second in adding certain rows and columns.

\begin{remark}{\rm
 ${\cal D}_{2t}$ is trivially embedded as a subgroup of ${\cal D}_{4t}$, the dihedral group of $4t$ elements. Concretely, if ${\cal D}_{4t}=
 \langle a,b\colon \,a^{2t}=b^2=(ab)^2=1\rangle$ then ${\cal D}_{2t}\cong \langle a^2,b\rangle\subset  {\cal D}_{4t}$.}
\end{remark}

\begin{proposition}
Let $M_\psi$ be a cocyclic matrix over ${\cal D}_{4t}$, then the $2t$ by $2t$ matrix obtained by eliminating from $M_\psi$ the rows and columns indexed with an even number is a cocyclic matrix over ${\cal D}_{2t}$ and we denote it as $\tilde{M}_{{\psi}}$.
\end{proposition}
\begin{proof}
On the one hand, taking into account the ordering fixed  above, ${\cal D}_{4t}=\{1,a,\ldots,a^{2t-1},b,ab,\ldots,a^{2t-1}b\}$, the rows and columns in $M_\psi$ indexed with an odd number correspond with $\{1,a^2,\ldots,a^{2t-2},b,a^2b,\ldots,a^{2t-2}b\}=\langle a^2,b\rangle$.

On the other hand,  if $\psi_{|_{\langle a^2,b\rangle}}$ denotes the restriction of $\psi$ to the subgroup $\langle a^2,b\rangle$, then $\psi_{|_{\langle a^2,b\rangle}}$ satisfies (\ref{condiciondecociclo}) for $G=\langle a^2,b\rangle\cong {\cal D}_{2t}$ since
$\psi$ satisfies (\ref{condiciondecociclo}) for $G={\cal D}_{4t}$. In other words, $\psi_{|_{\langle a^2,b\rangle}}$  is a cocycle for $G=\langle a^2,b\rangle$ and $\tilde{M}_{{\psi}}=M_{\psi_{|_{\langle a^2,b\rangle}}}$.
\end{proof}

\begin{lemma}\label{lemmapherc}
Let $M_{\psi_1}$ and $M_{\psi_2}$ be  cocyclic matrices over ${\cal D}_{4t}$ then
$$\tilde{M}_{\psi_1\cdot\psi_2}=\tilde{M}_{\psi_1}\circ \tilde{M}_{\psi_2},$$
where $M_{\psi_1\cdot\psi_2}=M_{\psi_1}\circ {M}_{\psi_2}$.
\end{lemma}
\begin{proof}
It is a straightforward consequence of the pointwise multiplication.
\end{proof}

From now on,  ${\cal B}=\{ \partial _2, \ldots , \partial_{4t-2},\beta_1,\beta_2,\gamma\}$ and
${\tilde{\cal B}}=\{ \tilde{\partial} _2, \ldots , \tilde{\partial}_{2t-1},\tilde{\beta}_2\}$ denote a  basis of cocycles for ${\cal D}_{4t}$ and ${\cal D}_{2t}$, respectively.

\begin{lemma}\label{lemmacobrc}
 The following identities hold:
$$\tilde{M}_{{\partial_i}}=\left\{\begin{array}{ll} J_{2t} & i \,\mbox{even}\\
   M_{\tilde{\partial}_{\frac{i+1}{2}}} & i\, \mbox{odd}\end{array}\right., \quad
   \tilde{M}_{{\beta_i}}=\left\{\begin{array}{cr} J_{2t} & i=1 \\
   M_{\tilde{\partial}_{\beta_{2}}} & i=2\end{array}\right.
$$
and
$$
\tilde{M}_{{\gamma}}=\prod_{i=1}^{ \frac{t-1}{2}}M_{ \tilde{\partial}_{2i}}\circ M_{\tilde{\partial}_{2t-2i+1}}.
$$

\end{lemma}
\begin{proof}
The identities above follow by direct inspection.
\end{proof}

 Given $M_\psi$ a  ${\cal D}_{4t}$-matrix. The following result describes, in terms of cocycles, the unique ${\cal D}_{2t}$-matrix obtained  by eliminating from $M_\psi$ the rows and columns indexed with an even number.

\begin{theorem}\label{cord2d4} Given a
${\cal D}_{4t}$-matrix
$$M_\psi=   M_{{\partial}_{2}}^{\alpha_2}\circ \ldots \circ  M_{{\partial}_{4t-2}}^{\alpha_{4t-2}}\circ
M_{\beta_1}^{k_1}\circ M_{\beta_2}^{k_2}\circ M_{\gamma}^{k_3}
$$
where  $(\alpha_2,\ldots,\alpha_{4t-2},k_1,k_2,k_3)$ denotes a concrete $4t$-upla with entries 0 or 1.
Then
$$\begin{array}{l}
\tilde{M}_\psi=   \tilde{M}_{{\partial}_{2}}^{\alpha_2}\circ \ldots \circ  \tilde{M}_{{\partial}_{4t-2}}^{\alpha_{4t-2}}\circ
\tilde{M}_{{\beta_1}}^{k_1}\circ \tilde{M}_{{\beta}_2}^{k_2}\circ \tilde{M}_{\gamma}^{k_3}\\[3mm]
\,\,\,\,\,\quad = \D\prod_{j=1}^{2t-2}M_{\tilde{\partial}_{j+1}}^{\alpha_{2j+1}}\circ M_{\tilde{\beta}_2}^{k_2}\circ \left(\prod_{i=1}^{ \frac{t-1}{2}}M_{ \tilde{\partial}_{2i}}\circ M_{\tilde{\partial}_{2t-2i+1}}\right)^{k_3}.
\end{array}$$
\end{theorem}
\begin{proof}
It follows from Lemmas \ref{lemmapherc} and \ref{lemmacobrc}.
\end{proof}

Now, fixed $M_{\tilde{\psi}}$ a ${\cal D}_{2t}$-matrix. The whole set of ${\cal D}_{4t}$-matrices constructed by  adding certain rows and columns to $M_{\tilde{\psi}}$ is provided in the next theorem.

\begin{theorem}\label{theaddrc}
Given a ${\cal D}_{2t}$-matrix
$$M_{\tilde{\psi}}=M_{\tilde{\partial}_{i_1}}\circ \ldots\circ M_{\tilde{\partial}_{i_w}}\circ M_{\tilde{\beta_2}}.$$
Then any ${\cal D}_{4t}$-matrix of the form:
$$M_\phi=  \prod_{i=1}^{2t-1} M_{{\partial}_{2i}}^{\alpha_i}\circ M_{\beta_1}^{k_1}\circ M_{\gamma}^{k_2}\circ \left(\prod_{i=1}^{ \frac{t-1}{2}}M_{ {\partial}_{4i-1}}\circ M_{{\partial}_{4t-4i+1}}\right)^{k_2}\circ\,M_{{\partial}_{2i_1-1}}\circ \ldots\circ M_{{\partial}_{2i_w-1}}\circ M_{{\beta_2}},$$
with $\alpha_i,k_1$ and $k_2$ taking values from $\{0,1\}$, satisfies that $\tilde{M}_{\phi}=M_{\tilde{\psi}}$.
\end{theorem}
\begin{proof}
Using Theorem \ref{cord2d4}, it is easy to check that $\tilde{M}_{\phi}=M_{\tilde{\psi}}$.
\end{proof}

Using our algebraic formalism we are able to  give a method for embedding a ${\cal D}_{2t}$-matrix in the rows and columns indexed with an odd number of a ${\cal D}_{4t}$-Hadamard matrix whenever it is possible.
 Although, theoretically this method provides a solution,  from the practical perspective it is only appropriate for numerical calculations in low orders, because in the worst case it needs to check $2^{2t}$ possibilities. Hence, properties providing some cut down in  complexity and some heuristic are generally needed.

\vspace{3mm}

\begin{algorithm} \label{alec2te4t} {\rm Embedding ${\cal D}_{2t}$-matrices in  ${\cal D}_{4t}$-Hadamard matrices.

\vspace{2mm}

\noindent{Input: a  ${\cal D}_{2t}$-matrix. $M_{\tilde{\psi}}=M_{\tilde{\partial}_{i_1}}\circ \ldots\circ M_{\tilde{\partial}_{i_w}}\circ M_{\tilde{\beta_2}}$.}
\newline
\noindent{Output: a ${\cal D}_{4t}$- Hadamard matrix $M_{\phi}$ which contains embedded $M_{\tilde{\psi}}$ (that is,
$\tilde{M}_\phi=M_{\tilde{\psi}}$),  if such  matrix exists. }

\vspace{3mm}
\begin{enumerate}
\item[] Step 1. Calculate $M_{\psi}=M_{{\partial}_{2i_1-1}}\circ \ldots\circ M_{{\partial}_{2i_w-1}}\circ M_{{\beta_2}}$.
\item[] Step 2. Calculate all possible combinations
$$M_\phi=  \prod_{i=1}^{2t-1} M_{{\partial}_{2i}}^{\alpha_i}\circ M_{\gamma}^{k}\circ \left(\prod_{i=1}^{ \frac{t-1}{2}}M_{ {\partial}_{4i-1}}\circ M_{{\partial}_{4t-4i+1}}\right)^k\,\circ\,M_{\psi}$$
where $k$ and $\alpha_{i}$ may take the values $0$ or $1$.
\item[] Step 3. If there exists a $(\alpha_1,\ldots,\alpha_{2t-1},k)$ such that $RE(M_\phi)=0$, then $M_\phi$ is Hadamard. Otherwise, such matrix does not exist for $M_{\tilde{\psi}}$.

\end{enumerate}
}
\end{algorithm}

\vspace{1.5mm}

\noindent{\sc Verification:}
By construction, $M_\phi$ is a cocyclic Hadamard matrix over $D_{4t}$ and $\tilde{M}_{\phi}=M_{\tilde{\psi}}$ (see Theorem \ref{theaddrc}).

\vspace{3mm}

In the sequel, we describe an algorithm looking for ${\cal D}_{2t}$-matrices with large determinant embedded in a Ha\-da\-mard matrix. Actually, the output $M_\psi$ is a Ha\-da\-mard matrix where the  ${\cal D}_{2t}$-matrix, $M_{\tilde{\psi}}$, is embedded in a such a way that $\tilde{M}_\psi=M_{\tilde{\psi}}$ . We give two strategies that are both based on exhaustive searches. The first one needs to calculate ${\cal D}_{2t}$-matrices with large determinant (by the approach given in \cite{AAFG11}) and use Algorithm \ref{alec2te4t} for detecting if  it is embedded in a Hadamard matrix. The second one needs to calculate  ${\cal D}_{4t}$-Hadamard matrices, $M_\psi$, (using \cite{AAFR09}) and check the determinant of  $\tilde{M_\psi}$.

\begin{algorithm} \label{algbusexh}
{\rm
Search for ${\cal D}_{4t}$-Hadamard matrices where a  ${\cal D}_{2t}$-matrix with large determinant exists embedded in it.

\vspace{3mm}

\noindent{Input:  an odd positive integer $t$.}
\newline
\noindent{Output: a ${\cal D}_{4t}$-Hadamard matrix $M_{\psi}$ where $\frac{det(\tilde{M}_{{\psi}})}{(4t-2)(2t-2)^{t-1}}\geq \kappa$,  if such  matrix exists, with $0.85\leq\kappa \leq 1$.}

\vspace{1.5mm}

\noindent{\em Approach 1}

\noindent{$\Omega\leftarrow \emptyset$}


\noindent{$\cal{S}\leftarrow$ The  list of ${\cal D}_{2t}$-matrices satisfying that  $\frac{det({M}_{\tilde{\psi}})}{(4t-2)(2t-2)^{t-1}}\geq \kappa$}


\noindent{while $\cal{S}$ is not empty $\{$ }
  \hspace*{0.5cm}\begin{enumerate}
  \item[] 1. Choose a matrix $M_{\tilde{\psi}}$ in $\cal{S}.$

  2. $\,{\cal{S}}\leftarrow{\cal{S}}\setminus\{M_{\tilde{\psi}}\}$.

   3. Check (using Algorithm \ref{alec2te4t}) whether exists a ${\cal D}_{4t}$-Hadamard matrix $M_\psi$ such that  $\tilde{M}_\psi=M_{\tilde{\psi}}$.
                      If not, go to 1; otherwise $\Omega\leftarrow M_\psi$, $S=\emptyset$.

   4. End while.

           \end{enumerate}
                   $\}$
\hspace{1mm}
\noindent{$\Omega$ }

\vspace{1mm}

\noindent{\em Approach 2}

\noindent{$\Omega\leftarrow \emptyset$}


\noindent{$\cal{S}\leftarrow$ The  list of ${\cal D}_{4t}$-Hadamard matrices}


\noindent{while $\cal{S}$ is not empty $\{$ }
  \hspace*{0.5cm}\begin{enumerate}
  \item[] 1. Choose a matrix $M_{{\psi}}$ in $\cal{S}.$

  2. $\,{\cal{S}}\leftarrow{\cal{S}}\setminus\{M_{{\psi}}\}$.

   3. Check whether $\frac{det(\tilde{M}_\psi)}{{(4t-2)(2t-2)^{t-1}}}\geq \kappa$.
                      If not, go to 1; otherwise $\Omega\leftarrow M_\psi$, $S=\emptyset$.

   4. End while.

           \end{enumerate}
                   $\}$
\hspace{1mm}
\noindent{$\Omega$ }

}
\end{algorithm}


\begin{center}\begin{table}\label{tabla1}$$\begin{array}{|c||c|c|c|c|c|} \hline t& \#({M}_\psi) & det(\tilde{M}_\psi)/2^{2t-1} & \#{M}_\psi({M}_{\tilde{\psi}}) &  RE & R\\
\hline 3& 72&\begin{array}{r}{\bf 5}\\4 \end{array} &
\begin{array}{l} 36(6)\\36(9)\end{array} & \begin{array}{c} 4\\ 4 \end{array} & \begin{array}{l}
1\\0.8\end{array}\\
\hline 5& 1400 & \begin{array}{r} {\bf 144}\\125\\
81 \end{array} & \begin{array}{l} 100(25)\\1200(100)\\
100(50) \end{array} &\begin{array}{c} 8\\ 8 \\ 8 \end{array} & \begin{array}{l} 1\\ 0.86\\
  0.56\end{array}
\\ \hline 7&11368& \begin{array}{r}{\bf 9477}\\ 8405 \\ 7569\\ 4096  \\
2197\\845\\ 841\\576\end{array}
& \begin{array}{l}
392(196)\\2352(294)\\2352(294)\\392(196)\\
1176(294)\\1764(294)\\ 1176(294)\\1764(147)
\end{array}&
\begin{array}{c}
12\\12\\12\\12\\
20\\20\\20\\20
\end{array}
&
\begin{array}{l}
 1\\0.887\\
 0.799\\0.432\\
 0.232\\0.0892\\
 0.0887\\0.061
\end{array}
\\ \hline \end{array}$$\caption{Determinant spectrum for ${\cal D}_{2t}$-matrices embedded in Hadamard ${\cal D}_{4t}$-matrices, $t=3,5,7$.}\end{table}\end{center}

\begin{center}\begin{table}\label{tabla2}$$\begin{array}{|c||c|c|c|c|c|} \hline t& \#({M}_\psi) & det(\tilde{M}_\psi)/2^{2t-1} & \#{M}_\psi({M}_{\tilde{\psi}}) &  RE & R\\
 \hline 9&130248&
\begin{array}{r}
1003520\\998001\\
950480\\ 912925\\
842724\\ 812500\\
426320\\
411892\\
372368\\
300713\\263169\\
240448\\
201977\\
198005\\186624\\
179776\\155236\\
126025\\
90593\\88445\\59049\\ 20800\\
4096\\0
\end{array}&
\begin{array}{l}
5832(972)\\7776(972)\\
13608(4374)\\
13608(1944)\\
7776(1944)\\
5832(972)\\
5832(2430)\\1944 (972)\\
5832 (486)\\3888 (972)\\7776 (972)\\
3888 (486)\\1944 (486)\\3888 (972)\\
7776 (972)\\1944(486)\\3888 (1944)\\
5832 (972)\\972 (486)\\1944 (972)\\5832 (162)\\
5832(486)\\
972 (243)\\5832 (486)
\end{array}
&\begin{array}{c}
16\\
16\\
16\\
16\\
16\\
16\\
24\\
24\\
24\\
24\\
24\\
24\\
24\\
24\\
24\\
24\\
24\\
24\\
32\\24\\
32\\24\\
32\\
32
\end{array}  & \begin{array}{l}
0.9007\\0.896\\0.853\\ 0.819 \\ 0.756\\0.729\\0.383\\0.370\\0.334\\ 0.27\\0.236\\0.216\\0.181\\0.178\\0.168\\0.161\\0.139\\
0.113\\0.081\\0.079 \\0.053\\0.019 \\0.004\\0
\end{array}
\\ \hline \end{array}$$\caption{Determinant spectrum for ${\cal D}_{18}$-matrices embedded in Hadamard ${\cal D}_{36}$-matrices.}\end{table}\end{center}

For $3\leq t\leq 9$ odd, we have performed an exhaustive search looking for  the possible values of the determinant of ${\cal D}_{2t}$-matrices embedded in ${\cal D}_{4t}$-Hadamard matrices using Algorithm \ref{algbusexh} with $0\leq \kappa\leq 1$, and have been displayed in Tables 2.1 and 2.2.
All the calculations have been worked out in {\sc Mathematica 4.0}, running on a {\em Pentium IV 2.400 Mhz DIMM DDR266 512
MB}.

In Tables 2.1 and 2.2,
 $\#({M}_\psi)$  denotes the number of ${\cal D}_{4t}$-Hadamard matrices (with representative cocycle  $M_{\beta_2}\circ M_\gamma$). The second column of the table shows the different values for the determinant of $\tilde{M}_\psi$, whereas the third   displays about the frequency of  ${\cal D}_{4t}$-Hadamard matrices with  the same value for $det(\tilde{M}_\psi)$ and   between  parenthesis is the number of different ${\cal D}_{2t}$-matrices embedded, the fourth column gives the  absolute row excess of $\tilde{M}_\psi$.  Finally, the last column informs about $R=\frac{det(\tilde{M}_\psi)}{(4 t - 2)(2 t - 2)^{t - 1}}$ (which is called {\em efficiency of the design} in \cite{SXKM03}).

 Some interesting properties can be observed in these tables. First, for $t=3,5,7$ and $9$ there is just one value of $R$ above $0.9$
 (when $R=1$ we got D-optimal designs) and some of them above $0.85$. Second, for every pair of Hadamard matrices ${M}_{\psi_1}$ and ${M}_{\psi_2}$ with $det(\tilde{M}_{\psi_2})=det(\tilde{M}_{\psi_2})$ satisfy that  $RE(\tilde{M}_{\psi_1})=RE(\tilde{M}_{\psi_2})$.
 Third, the minimum $RE$  corresponds with the largest values for the determinant for  $\tilde{M}_\psi$. Taking into account this last property, a result about   characterizing the ${\cal D}_{2t}$-matrices with minimum RE has interest. This characterization will be given here in terms of {\em maximal $n$-paths}, for this terminology we refer to \cite{AAFR08}.


\begin{proposition}
Let $M_{\tilde{\psi}}$ be a ${\cal D}_{2t}$-matrix. The $RE(M_{\tilde{\psi}})=2t-2$ iff
\begin{enumerate}
\item $M_{\tilde{\psi}}$ decomposes as a combination of the form  $M_{\tilde{\partial}_{i_1}}\circ \ldots\circ M_{\tilde{\partial}_{i_w}}\circ M_{\tilde{\beta_2}}$.
\item  The number of maximal $s$-paths for $M_{\tilde{\psi}}$ is either $\frac{t-1}{2}$ or   $\frac{t+1}{2}$ for $s=2,..,t$.
\end{enumerate}
\end{proposition}
\begin{proof}
 The proof of this result follows just from the study of the distribution of $-1$ by rows in the elementary coboundaries (see \cite{AAFR08}). Actually, considering a matrix $N=M_{\tilde{\partial}_{i_1}}\circ \ldots\circ M_{\tilde{\partial}_{i_w}}$,  Remark \ref{notacob} implies that there is  (necessarily) an even positive number $2f_s$ of $-1$s located at row $s$ of $N$, $2 \leq s \leq 2t$.
Hence,
$$RE(N)\geq 4t-2.$$

For $t+1\leq s \leq 2t$, taking into account  the presentation of $D_{2t}$, it may be readily checked that $(a^kb)^{-1}=a^kb$. In this circumstance, the  $-1$s entries located at row $s$
are distributed in such a way that precisely $f_s$ of them occur through columns 1 to $t$, whereas the remaining $f_s$ occur through columns $t+1$ to $2t$. Furthermore, focusing on row $s$,  any two coboundary matrices $M_{\tilde{\partial} _i}$ and $M_{\tilde{\partial} _j}$ either share their two $-1$s entries at row $s$ (just in a unique row),  or do not share any of them at all $s$.

Consequently, attending to the form of $M_{\tilde{\beta}_2}$, the summation of row $s$ of $M_{\tilde{\psi}}$ is $0$ iff  $M_{\tilde{\psi}}=M_{\tilde{\partial}_{i_1}}\circ \ldots\circ M_{\tilde{\partial}_{i_w}}\circ M_{\tilde{\beta_2}}$.

For $2\leq s\leq t$, attending to the form of $M_{\tilde{\beta}_2}$ and the fact that  for every row $s$ every maximal $s$-path determines two $-1$'s entries in $N$  (see \cite{AAFR08}). Then, the summation of row $s$ of $_{\tilde{\psi}}$ is 2 or $-2$ iff the number of maximal $s$-paths for $M$ is either $\frac{t-1}{2}$ or   $\frac{t+1}{2}$ for $s=2,..,t$.

From above two equivalences, we conclude with the desired result.
\end{proof}

\section{On the pivot structure of ${\cal D}_{20}$-Hadamard matrices}

In 1988 Day and Peterson \cite{DP88} proved that the growth factor of an $n\times n$ Sylvester Hadamard matrix is $n$. In \cite[Example 6.2.4]{Hor07} it is shown  that the Sylvester Hadamard matrix of order $2^n$ is cocylic over $\integers_2^n$. Taking into account the difficulty of proving Cryer's conjecture and the previous statements,  the following question arises in a natural way: Is the growth factor of a ${\cal D}_{4t}$-Hadamard matrix its order?  

In this section, we will focus on  ${\cal D}_{20}$-Hadamard matrices and  will extract some information on their pivot patterns from the D-optimal designs (and other specific submatrices with concrete determinants) that exist embedded in these matrices.

It is well-known    that there is a unique D-optimal design up to equivalence in each order up to $n=10$ \cite{KO06}.

  In Table 2.1,   the $D$-optimal design of order 10 appears (see \ref{d10}) embedded in  100 different  ${\cal D}_{20}$-Hadamard matrices. Using Magma V2.11 we found just one inequivalent ${\cal D}_{20}$-Hadamard matrix among the 100 (in higher orders, see \cite{Arm10} for an inequivalence criterion for ${\cal D}_{4t}$-Hadamard matrices). Based on exhaustive searches for minors of order 10 embedded in Hadamard matrices of order 20, we also identified the $D$-optimal design of order 10  in  the other two inequivalent Hadamard matrices of order 20. So, we can state the following result:

\newpage

\begin{lemma}
The D-optimal design of order $10$ is embedded in all Hadamard matrices of order 20.
\end{lemma}

\begin{corollary}
For a Hadamard matrix of order 20, it holds that $H(10)$ can take the value $144\cdot 2^9$.
\end{corollary}

\begin{equation}\label{d10}
D_{10}=\left[\begin{array}{cccccccccc}
 - & 1 & 1 & 1 & 1 & -& 1 & 1 & 1 & 1\\
 1 & 1 & - & 1 & 1 & 1 & - & 1 & 1 & 1\\
 1 & - & 1 & 1 & 1 & 1 & 1 & 1 & - & 1\\
 - & - & - & 1 & - & 1 & 1 & 1 & 1 & -\\
 - & - & - & - & 1 & 1 & 1 & - & 1 & 1\\
 1 & - & - & 1 & 1 & - & 1 & - & 1 & -1\\
 - & - & 1 & 1 & 1 & 1 & - & - & 1 & -1\\
 1 & 1 & 1 & 1 & - & 1 & 1 & - & 1 & 1\\
 - & 1 & - & 1 & 1 & 1 & 1 & - & - & -\\
 1 & 1 & 1 & - & 1 & 1 & 1 & 1 & 1 & -
 \end{array}\right]
\end{equation}

\begin{proposition}
Given $H$ a Hadamard matrix of order 20, there exists an equivalent CP matrix $H'$ where   $D_{10}$ appears located at the top left position in $H'$.
\end{proposition}
\begin{proof}
Let $\hat{H}$ be an equivalent matrix to $H$ such that $D_{10}$ appears located at the top left position in $\hat{H}$. Notice that $D_{10}$ (see (\ref{d10}) is feasible.
Our aim is to prove that  $D_{10}(i)=\hat{H}'(i), \,\,i=1,\ldots,10$.
\begin{itemize}
\item If $i=1, 2, 3, 4, 5, 6 $ and $10$ then the identity  holds because $D_{10}(i)$ takes the maximum determinant of any $(-1,1)$-matrix of order $i$ (see \cite{OS05}).
\item It well-known that $D_6\notin D_7$ (\cite{SM10,Mit11}), hence the  maximum determinant of any $(-1,1)$-matrix, $A$,  where   $D_{6}$ appears located at the top left position, is lesser than or equal to 512 (the second greatest, see the spectrum of the determinant function \cite{OS05}).
    Due to $D_{10}(7)=512$ then $D_{10}(7)=H'(7)$.
\item  We extend the $7\times 7$ matrix $D_{10}(7\times 7)$ located at the top left position in $D_{10}$ to all possible $8\times 8$ matrices of the form
$$\left[\begin{array}{cccccccc}
 - & 1 & 1 & 1 & 1 & -& 1 &  1\\
 1 & 1 & - & 1 & 1 & 1 & - &  *\\
 1 & - & 1 & 1 & 1 & 1 & 1 & * \\
 - & - & - & 1 & - & 1 & 1 & *\\
 - & - & - & - & 1 & 1 & 1 & *\\
 1 & - & - & 1 & 1 & - & 1 & *\\
 - & - & 1 & 1 & 1 & 1 & - & *\\
  1 & * & * & * & * & * & * & *\\
  \end{array}\right]$$
where the elements $*$ can be $\pm 1$. Computing the determinant of these $2^{13}$ $8\times 8$ matrices, we have that it is always lesser than or equal to
$D_{10}(7)=2560$. Therefore, $D_{10}(8)=H'(8)$.
\item Taking into account that $D_9\notin D_{10}$, then the  maximum $9\times 9$ minor of $D_{10}$ is lesser than or equal to 12288 (the second greatest, see the spectrum of the determinant function \cite{OS05}).
    Due to $D_{10}(9)=12288$ then $D_{10}(9)=H'(9).$
    \end{itemize}
Now, taking $H'$ as the CP extension of $\hat{H}$, the proof follows.
\end{proof}
\begin{corollary}
The pivot pattern $(1,2,2,4,3,10/3,16/5,5,24/5,6)$ for the first ten pivots appears in the three classes for Hadamard matrices of orden 20.
\end{corollary}

In a recent search that we have performed, we have found  this result:
\begin{itemize}
\item There is a CP matrix $H$  equivalent to the following   ${\cal D}_{20}$-Hadamard matrix \newline $M_\psi=M_{\partial_2}\circ M_{\partial_4}\circ M_{\partial_8}\circ M_{\partial_{10}}\circ M_{\partial_{13}}\circ M_{\partial_{14}}\circ M_{\beta_2}\circ M_{\gamma}$ such that
$H(10)=125\cdot 2^9$.
\end{itemize}

 Although we didn't find any result in the literature asserting that if the existence of a submatrix with large determinant is proven for a  matrix $A$, then we can indeed assume that it always appears in the upper left corner for some CP matrix $A'$ equivalent to $A$. It seems to be true at least when this submatrix reaches the largest determinant (see \cite[p.1763]{Mit11}). Also, the above result agree this for other large determinant (the second largest in Table 2.1).

\section{Conclusions and further work}
In this paper we have described a method for embedding a ${\cal D}_{2t}$-matrix  in the rows and columns indexed with an odd number of a ${\cal D}_{4t}$-Hadamard matrix whenever it is possible.
If this ${\cal D}_{2t}$-matrix has a determinant attaining the largest possible value, we get a D-optimal design. This method relies on two combinatorial operations on a cocyclic matrix: eliminate and add certain rows and columns. The idea behind this approach has been to translate
these two combinatorial operations into  a pure algebraic framework (concretely, in terms of cocycles). Finally, our study has provided some information about the pivot values when Gaussian elimination with complete pivoting is performed on   $D_{20}$-Hadamard matrices.

Our next goals are:

\begin{itemize}

\item Study the relationship between the $RE$  and the values of the determinant for  ${\cal D}_{2t}$-matrices.
\item Design heuristic searches based on $RE$ for Algorithm \ref{algbusexh}.
\item Study if $\tilde{M}_\psi$ satisfies that $\frac{det(\tilde{M}_\psi)}{{(4t-2)(2t-2)^{t-1}}}\geq 0.85$ implies that a CP  matrix $M$ equivalent to $M_\psi$ exists such that $M(2t)=det(\tilde{M}_\psi)$.
 \item Study the pivot structure of ${\cal D}_{4t}$-Hadamard matrices.
 \item Design an ``efficient'' method to construct ${\cal D}_{4t}$-Hadamard matrices from ${\cal D}_{2t}$-matrices.

\end{itemize}



\bigskip
{\bf Acknowledgment.} This work has been partially supported by the research projects FQM-016 and P07-FQM-02980 from JJAA and MTM2008-06578 from MIC\-INN (Spain) and FEDER (European Union). The authors would also like to thank Kristeen Cheng for her reading of this article.


\end{document}